\documentclass[12pt,english]{article}
\usepackage[T1]{fontenc}
\usepackage[latin1]{inputenc}
\usepackage{textcomp}
\usepackage{amssymb}

\makeatletter

\makeatletter\@addtoreset {equation}{section}\makeatother

\newtheorem{theorem}{Theorem}
\newtheorem{lemma}{Lemma}
\newtheorem{remark}{Remark}

\newtheorem{corollary}{Corollary}

\newenvironment{proof}{
    \noindent {\it Proof.}}{\hfill$\Box$
}
\newenvironment{proof1}{
    \noindent {\it Proof }}{\hfill$\Box$
}

\usepackage{babel}
\makeatother

\begin{document}

\title{\textbf{Global well-posedness of the short-pulse and sine--Gordon
equations in energy space}}

\author{Dmitry Pelinovsky and Anton Sakovich \date{}\\
 {\small Department of Mathematics, McMaster University, Hamilton,
Ontario, Canada, L8S 4K1}}

\maketitle
\begin{abstract}
We prove global well-posedness of the short-pulse equation with small
initial data in Sobolev space $H^{2}$. Our analysis relies on local
well-posedness results of Schäfer \& Wayne \cite{SW04}, the correspondence
of the short-pulse equation to the sine--Gordon equation in characteristic
coordinates, and a number of conserved quantities of the short-pulse
equation. We also prove local and global well-posedness of the sine--Gordon
equation in an appropriate function space. 
\end{abstract}

\section{Introduction}

Short-pulse approximations of nonlinear wave packets in dispersive
media were considered recently with various analytical techniques
(see, e.g., \cite{CL08} for a review of results). The previously
known model for small-amplitude quasi-harmonic pulses, the nonlinear
Schrödinger equation, is replaced in the short-pulse approximation
by a new set of nonlinear evolution equations. Among these models,
we consider the model derived by Schäfer \& Wayne \cite{SW04} for
short pulses in nonlinear Maxwell's equations. Chung {\em et al.}
\cite{CJSW05} justified derivation of this model in the linear case
and presented numerical approximations of modulated pulse solutions.
This model dubbed as {\em the short-pulse equation} can be conveniently
expressed in the normalized form by \begin{equation}
u_{xt}=u+\frac{1}{6}\left(u^{3}\right)_{xx},\label{short-pulse}\end{equation}
 where $u(x,t):\mathbb{R}\times\mathbb{R}_{+}\mapsto\mathbb{R}$ and
the subscript denotes partial derivatives. In addition to the derivation
of the short-pulse equation, the pioneer paper \cite{SW04} contains
two mathematical results. First, non-existence of a smooth traveling
wave solution was proved in the entire range of the speed parameter.
Second, local well-posedness was proven in space $H^{s}$ for an integer
$s\geq2$, where $H^{s}$ denotes the standard Sobolev space with
the norm \[
\|f\|_{H^{s}}=\left(\int_{\mathbb{R}}\left(1+|k|^{2}\right)^{s}|\hat{f}(k)|^{2}dk\right)^{1/2}\]
 and $\hat{f}(k):=\frac{1}{\sqrt{2\pi}}\int_{\mathbb{R}}f(x)e^{-ikx}dx$
is the Fourier transform of $f(x)$ on a real axis.

The first result was recently extended by Costanzino, Manukian, \&
Jones \cite{CMJ08}, who proved existence of smooth traveling solutions
in the regularized short-pulse equation, \begin{equation}
u_{xt}=u+\frac{1}{6}\left(u^{3}\right)_{xx}+\beta u_{xxxx},\label{Ostrovsky-cubic}\end{equation}
 for a small positive $\beta>0$. They have also derived the regularized
short-pulse equation (\ref{Ostrovsky-cubic}) in the context of the
nonlinear Maxwell equation with a high-frequency dispersion. To construct
homoclinic solutions with slow and fast motions, the authors of \cite{CMJ08}
applied the Fenichel theory for singularly perturbed differential
equations.

In this paper, we are extending the second result of \cite{SW04},
namely we prove global well-posedness of the short-pulse equation
(\ref{short-pulse}) in $H^{2}$. The problem of global well-posedness
has been studied in a number of recent publications \cite{VL04,LM06,GL07}
in the context of a similar Ostrovsky equation \begin{equation}
u_{xt}=u+\left(u^{2}\right)_{xx}+\beta u_{xxxx},\label{Ostrovsky-quadratic}\end{equation}
 which models small-amplitude long waves in a rotating fluid. Clearly,
the regularized short-pulse equation (\ref{Ostrovsky-cubic}) can
be considered as the Ostrovsky equation with a cubic nonlinearity,
so that many of the well-posedness results can be equally applied
to both (\ref{Ostrovsky-cubic}) and (\ref{Ostrovsky-quadratic}).
Varlamov \& Liu \cite{VL04} proved local well-posedness of (\ref{Ostrovsky-quadratic})
in space $H^{s}\cap\dot{H}^{-1}$ for $s>\frac{3}{2}$, where $\dot{H}^{-1}$
is defined by the norm \[
\|f\|_{\dot{H}^{-1}}=\left(\int_{\mathbb{R}}\frac{|\hat{f}(k)|^{2}}{k^{2}}dk\right)^{1/2}.\]
 If $f\in\dot{H}^{-1}$, then the constraint $\hat{f}(0)=0$ holds,
which is written in physical space as $\int_{\mathbb{R}}f(x)dx=0$.
Under the constraint above, we can define an operator $\partial_{x}^{-1}f$
in any of the equivalent forms \[
\partial_{x}^{-1}f:=\int_{-\infty}^{x}f(x')dx'=-\int_{x}^{\infty}f(x')dx'=\frac{1}{2}\left(\int_{-\infty}^{x}-\int_{x}^{\infty}\right)f(x')dx',\]
 so that $\|f\|_{\dot{H}^{-1}}=\|\partial_{x}^{-1}f\|_{L^{2}}$. The
space $H^{1}\cap\dot{H}^{-1}$ is the energy space of the Ostrovsky
equation (\ref{Ostrovsky-quadratic}), where the power $V(u)=\|u\|_{L^{2}}^{2}$
and the energy \[
E(u)=\int_{\mathbb{R}}\left(\beta(\partial_{x}u)^{2}+\frac{1}{2}\left(\partial_{x}^{-1}u\right)^{2}-\frac{1}{3}u^{3}\right)dx\]
 conserve in time $t$. Using conserved quantities and local existence
in $H^{1}\cap\dot{H}^{-1}$, Linares \& Milanes \cite{LM06} and Gui
\& Liu \cite{GL07} proved global well-posedness of the Ostrovsky
equation (\ref{Ostrovsky-quadratic}) in the energy space. However,
their proof is only valid for $\beta>0$ and it is not applicable
to the short-pulse equation (\ref{short-pulse}).

Using local well-posedness results of Schäfer \& Wayne \cite{SW04}
and conserved quantities of the short-pulse equation (\ref{short-pulse})
found by Brunelli \cite{Br05}, we shall prove our main theorem on
global well-posedness of the short-pulse equation.

\begin{theorem} \label{theorem-wellposedness} Assume that $u_{0}\in H^{2}$
and \[
\|u_{0}'\|_{L^{2}}^{2}+\|u_{0}''\|_{L^{2}}^{2}<1.\]
 Then the short-pulse equation (\ref{short-pulse}) admits a unique
solution $u(t)\in C(\mathbb{R}_{+},H^{2})$ satisfying $u(0)=u_{0}$.
\end{theorem}

Our main motivation of interest in global well-posedness of the short-pulse
equation (\ref{short-pulse}) begins in the recent discovery of exact
modulated pulse solutions, which were modeled numerically in \cite{CJSW05}.
The exact solutions were found by Sakovich \& Sakovich \cite{SS06},
who also showed in \cite{SS05,SS07} integrability of the short-pulse
equation (\ref{short-pulse}) and its equivalence to the sine--Gordon
equation in characteristic coordinates \begin{equation}
w_{yt}=\sin(w),\label{sine-Gordon}\end{equation}
 where $w(y,t):\mathbb{R}\times\mathbb{R}_{+}\mapsto\mathbb{R}$ is
a new function in a new independent variable $y$. According to the
results in \cite{SS07}, there exists a local transformation from
solutions of the sine--Gordon equation (\ref{sine-Gordon}) to solutions
of the short-pulse equation (\ref{short-pulse}) given by $u(x,t)=w_{t}(y(x,t),t)$,
where $y=y(x,t)$ is found by inverting the function $x=x(y,t)$ from
solutions of the system of partial differential equations \begin{equation}
\frac{\partial x}{\partial y}=\cos(w),\quad\frac{\partial x}{\partial t}=-\frac{1}{2}\left(w_{t}\right)^{2}.\label{transformation}\end{equation}
 The compatibility of system (\ref{transformation}) results in the
sine--Gordon equation (\ref{sine-Gordon}). The function $x=x(y,t)$
is invertible with the inverse $y=y(x,t)$ for a fixed $t\in\mathbb{R}$
if $w(\cdot,t)$ belongs to the space \begin{equation}
H_{c}^{s}=\left\{ w\in H^{s}:\quad\|w\|_{L^{\infty}}\leq w_{c}<\frac{\pi}{2}\right\} ,\quad s>\frac{1}{2}.\label{space-Hs-c}\end{equation}

To understand the long-term dynamics of perturbations near the exact
modulated pulse solutions (which are dubbed as {\em breathers}
in the context of the sine--Gordon equation), we need to prove local
and global existence of solutions of the sine--Gordon equation (\ref{sine-Gordon})
in $H_{c}^{s}$ for $s>\frac{1}{2}$. If the constraint in $H_{c}^{s}$
can be kept globally in time, these results would imply that a solution
of the short-pulse equation (\ref{short-pulse}) starting with small
data $u_{0}\in H^{2}$ will remain bounded in $H^{2}$ for all $t\in\mathbb{R}_{+}$.
Together with integrability of the short-pulse equation (\ref{short-pulse}),
these results may suggest orbital or asymptotic stability of the modulated
pulse solutions. The latter problem is, however, beyond the scope
of the present paper.

The sine-Gordon equation in characteristic coordinates was considered
long ago by Kaup \& Newell \cite{KN78} using formal applications
of the stationary phase method. Local well-posedness of this equation
is a non-trivial problem due to the presence of the constraint \begin{equation}
\int_{\mathbb{R}}\sin(w(y,t))dy=0,\qquad\forall t\in\mathbb{R}_{+},\label{constraint}\end{equation}
 which gives a necessary but not sufficient condition that the solution
$w(\cdot,t)$ stays in $H_{c}^{s}$ for all $t\in\mathbb{R}_{+}$.
Our treatment of this equation is rigorous and we shall prove that
the sine--Gordon equation (\ref{sine-Gordon}) is locally well-posed
in space $H^{s}\cap\dot{H}^{-1}$ for an integer $s\geq1$ in the
new variable $q=\sin(w)$. Global well-posedness is proved in $H^{1}\cap\dot{H}^{-1}$
with the help of three conserved quantities of the sine--Gordon equation
(\ref{sine-Gordon}). The result can be extended in $H^{s}\cap\dot{H}^{-1}$
for an integer $s\geq2$ if more conserved quantities are incorporated
into analysis.

The sine--Gordon equation in the laboratory coordinates \[
w_{\tau\tau}-w_{\xi\xi}=\sin(w)\]
 is known to be locally well-posed in a weaker space $L^{p}$ for
$p\geq2$, see Appendix B of Buckingham \& Miller \cite{BM08}. Similarly
to this work, our analysis is also based on the method of Picard iterations
to prove local well-posedness of the sine--Gordon equation in characteristic
coordinates (\ref{sine-Gordon}). These results provide an alternative
proof of the local well-posedness theorem for the short-pulse equation
(\ref{short-pulse}), thanks to the transformation (\ref{transformation}).
Our treatment of the problem is, however, simpler than the original
method of \cite{SW04}, where modified Picard iterations were constructed
using solutions of quasi-linear hyperbolic equations along the characteristics.
Moreover, additional properties of local solutions to the two equations
are identified via the transformation (\ref{transformation}) and
these properties are found to be useful in rigorous treatment of conserved
quantities for the two equations.

The paper is organized as follows. Section 2 deals with local well-posedness
of the sine--Gordon equation (\ref{sine-Gordon}). Correspondence
of local solutions is established in Section 3. Section 4 gives the
proof of Theorem \ref{theorem-wellposedness} on global well-posedness
of the short-pulse equation (\ref{short-pulse}). Global well-posedness
of the sine--Gordon equation (\ref{sine-Gordon}) is proven in Section
5.

\section{Local well-posedness of the sine--Gordon equation (\ref{sine-Gordon})}

We shall consider solutions $w(y,t)$ of the sine--Gordon equation
(\ref{sine-Gordon}) vanishing at infinity $|y|\to\infty$ for $t\geq0$.
Therefore, we eliminate kink solutions of the sine--Gordon equation,
which connect different equilibrium states between multiplies of $2\pi$
at different infinities. Our reasoning is that the kink solutions
lead to non-invertible functions $x=x(y,t)$ with respect to $y$
in the transformation (\ref{transformation}) and give multi-valued
loop solutions of the short-pulse equation (\ref{short-pulse}) (see
\cite{SS06} for details). Not only we are considering localized solutions
$w(\cdot,t)$ in space $H^{s}$ for $s>\frac{1}{2}$, we need to control
\[
\|w(\cdot,t)\|_{L^{\infty}}\leq w_{c}<\frac{\pi}{2},\]
 to ensure that the transformation (\ref{transformation}) is invertible
at least for $t\in[0,T]\subset\mathbb{R}_{+}$ for some $T>0$. Thus,
we need to prove that the sine--Gordon equation (\ref{sine-Gordon})
admits a local solution $w(\cdot,t)$ in space $C([0,T],H_{c}^{s})$
for some $T>0$, where $H_{c}^{s}$ is defined by (\ref{space-Hs-c}).

To incorporate the constraint (\ref{constraint}) for solutions of
the sine--Gordon equation (\ref{sine-Gordon}), we introduce a new
variable $q:=\sin(w)$, so that the constraint (\ref{constraint})
becomes a linear constraint \begin{equation}
\int_{\mathbb{R}}q(y,t)dy=0.\label{constraint-linear}\end{equation}
 The sine--Gordon equation (\ref{sine-Gordon}) transforms to the
evolution equation \begin{equation}
q_{t}=\sqrt{1-q^{2}}\partial_{y}^{-1}q,\label{time-evolution-q}\end{equation}
 where the operator $\partial_{y}^{-1}$ acts on an element of $H^{s}$
under the constraint (\ref{constraint-linear}), so that \[
\partial_{y}^{-1}q:=\int_{-\infty}^{y}q(y',t)dy'=-\int_{y}^{\infty}q(y',t)dy'=\frac{1}{2}\left(\int_{-\infty}^{y}-\int_{y}^{\infty}\right)q(y',t)dy'.\]
 Let us introduce the nonlinear function \begin{equation}
f(q):=1-\sqrt{1-q^{2}}=\frac{q^{2}}{1+\sqrt{1-q^{2}}}\label{nonlinear-function}\end{equation}
 and write the Cauchy problem for equation (\ref{time-evolution-q})
in the equivalent form \begin{equation}
\left\{ \begin{array}{l}
q_{t}=(1-f(q))\partial_{y}^{-1}q,\\
q|_{t=0}=q_{0}.\end{array}\right.\label{Cauchy}\end{equation}
 The nonlinear function $f(q)$ is squeezed by the quadratic functions
\[
\forall|q|\leq1:\quad\frac{q^{2}}{2}\leq f(q)\leq q^{2},\]
 which allows us to interpret the term $f(q)\partial_{y}^{-1}q$ as
a nonlinear perturbation to the linear evolution induced by $\partial_{y}^{-1}q$.

To analyze the Cauchy problem for the nonlinear time evolution (\ref{Cauchy}),
we obtain information on the fundamental solution of the underlying
linear problem \begin{equation}
\left\{ \begin{array}{l}
Q_{t}=\partial_{y}^{-1}Q,\\
Q|_{t=0}=Q_{0}.\end{array}\right.\label{Cauchy-linear}\end{equation}
 Let us denote $L=\partial_{y}^{-1}$ and $Q(t)=e^{tL}Q_{0}$. The
solution operator $e^{tL}$ is a norm-preserving map from $H^{s}$
to $H^{s}$ for any $s\geq0$ in the sense of \begin{equation}
\|Q(t)\|_{H^{s}}=\|e^{tL}Q_{0}\|_{H^{s}}=\|Q_{0}\|_{H^{s}},\quad\forall t\in\mathbb{R},\label{norm-preservation}\end{equation}
 which follows from the Fourier transform \[
{\cal F}(e^{tL})=e^{-\frac{it}{k}},\quad k\in\mathbb{R},\]
 involving a bounded oscillatory integral with a singular behavior
as $k\to0$. The solution operator $e^{tL}$ can be represented by
the convolution integrals involving Bessel's function $J_{0}$ of
the first kind.

\begin{lemma} \label{lemma-bounds} Let $K_{t}(y)$ and $J_{t}(y)$
be defined by \[
K_{t}(y):=\sqrt{\frac{t}{y}}J_{0}'(2\sqrt{ty}),\quad J_{t}(y):=J_{0}(2\sqrt{ty}),\quad\forall t,y\in\mathbb{R}_{+}.\]
 There exists $C>0$ such that for all $t\in\mathbb{R}_{+}$ \[
\|K_{t}\|_{L^{\infty}}\leq Ct,\quad\|K_{t}\|_{L^{2}}\leq Ct^{1/2},\quad\|J_{t}\|_{L^{\infty}}\leq1,\]
 whereas $K_{t}\notin L^{1}$, $J_{t}\notin L^{1}$, and $J_{t}\notin L^{2}$.
\end{lemma}

\begin{proof} These properties follow from properties of Bessel function
$J_{0}(z)$, see, e.g., \cite{grafakos}. \end{proof}

It follows by direct substitution that the linear Cauchy problem (\ref{Cauchy-linear})
has a solution in the form \begin{equation}
Q(y,t)=Q_{0}(y)+\int_{y}^{\infty}K_{t}(y'-y)Q_{0}(y')dy',\quad(y,t)\in\mathbb{R}\times\mathbb{R}_{+},\label{solution-linear}\end{equation}
 which is bounded if $Q_{0}\in L^{\infty}\cap L^{1}$. In addition
to $Q(y,t)$, we shall also consider an integral of $Q(y,t)$ in $y$,
given by \begin{equation}
P(y,t):=-\int_{y}^{\infty}Q(y',t)dy'=-\int_{y}^{\infty}J_{t}(y'-y)Q_{0}(y')dy'.\label{formula-P}\end{equation}

The local well-posedness analysis is based on the integral equation
\begin{equation}
q(t)=Q(t)-\int_{0}^{t}e^{(t-t')L}f(q(t'))p(t')dt',\label{integral-form}\end{equation}
 which follows from Duhamel's principle for the nonlinear Cauchy problem
(\ref{Cauchy}). Here \[
q(t):=q(y,t),\quad p(t):=p(y,t)=-\int_{y}^{\infty}q(y',t)dy',\]
 $Q(t)=e^{tL}q_{0}$ is the solution of the linear problem (\ref{Cauchy-linear})
with $Q_{0}=q_{0}$, and $f(q)$ is defined by (\ref{nonlinear-function}).
We shall work in the space $X_{c}^{s}$ given by \begin{equation}
X_{c}^{s}=\left\{ q\in H^{s}\cap\dot{H}^{-1}:\;\|q\|_{L^{\infty}}\leq q_{c}<1\right\} ,\quad s>\frac{1}{2}.\label{space-X}\end{equation}
 Since $p_{y}=q$, it is clear that $p\in H^{s+1}$ if $p\in L^{2}$
and $q\in H^{s}$. We need to show that the vector field of the integral
equation (\ref{integral-form}) is a Lipschitz map in the vector space
$X^{s}:=H^{s}\cap\dot{H}^{-1}$ equipped with the norm \[
\|q\|_{X^{s}}:=\|q\|_{H^{s}}+\|p\|_{L^{2}}\]
 for any $t\in[0,T]$ and it is a contraction operator for a sufficiently
small $T>0$. This construction gives local well-posedness of the
sine--Gordon equation (\ref{sine-Gordon}).

\begin{theorem}\label{theorem-local-wellposedness} Assume that $q_{0}\in X_{c}^{s}$,
$s>\frac{1}{2}$. There exist a $T>0$ such that the Cauchy problem
(\ref{Cauchy}) admits a unique local solution $q(t)\in C([0,T],X_{c}^{s})$
satisfying $q(0)=q_{0}$. Moreover, the solution $q(t)$ depends continuously
on initial data $q_{0}$. \end{theorem}

\begin{proof} Fix $s>\frac{1}{2}$, $q_{c}\in(0,1)$, and $\delta\in(0,C_{s}^{-1})$,
where constant $C_{s}>0$ gives the upper bound of the Banach algebra
property \begin{equation}
\forall f,g\in H^{s}:\quad\|fg\|_{H^{s}}\leq C_{s}\|f\|_{H^{s}}\|g\|_{H^{s}}.\label{banach-algebra}\end{equation}
 We assume that $\|q_{0}\|_{X^{s}}\leq\alpha\delta$ and $\|q_{0}\|_{L^{\infty}}\leq\alpha q_{c}$
for a fixed $\alpha\in(0,1)$ and prove that there exists a $T>0$
such that $q(t)\in C([0,T],X_{c}^{s})$ is a solution of the Cauchy
problem (\ref{Cauchy}) such that $q(0)=q_{0}$, $\|q(t)\|_{X^{s}}\leq\delta$,
and $\|q(t)\|_{L^{\infty}}\leq q_{c}$ for all $t\in[0,T]$. To do
so, we find bounds on the supremum of $\|q(t)\|_{H^{s}}$, $\|q(t)\|_{\dot{H}^{-1}}$,
and $\|q(t)\|_{L^{\infty}}$ on $[0,T]$ from the solution of the
integral equation (\ref{integral-form}).

By the triangle inequality, the norm-preserving property (\ref{norm-preservation}),
and the Banach algebra property of $H^{s}$, we obtain \begin{eqnarray*}
\|q(t)\|_{H^{s}} & \leq & \|Q(t)\|_{H^{s}}+\int_{0}^{t}\|e^{(t-t')L}f(q(t'))p(t')\|_{H^{s}}dt'\\
 & \leq & \|q_{0}\|_{H^{s}}+C_{s}\int_{0}^{t}\|f(q(t'))\|_{H^{s}}\|p(t')\|_{H^{s}}dt'.\end{eqnarray*}
 To deal with nonlinear function $f(q)$, we expand it in the Taylor
series \[
\forall|q|<1:\quad f(q)=1-\sqrt{1-q^{2}}=\sum_{n=1}^{\infty}\frac{(2n-3)!!}{n!2^{n}}q^{2n},\]
 which involves only positive coefficients. By invoking again the
Banach algebra property, we obtain\[
\forall C_{s}\|q\|_{H^{s}}<1:\quad\|f(q)\|_{H^{s}}\leq\sum_{n=1}^{\infty}\frac{(2n-3)!!}{n!2^{n}}C_{s}^{2n-1}\|q\|_{H^{s}}^{2n}=\frac{1-\sqrt{1-C_{s}^{2}\|q\|_{H^{s}}^{2}}}{C_{s}}\leq C_{s}\|q\|_{H^{s}}^{2},\]
 thanks to the representation (\ref{nonlinear-function}). As a result,
we have\begin{equation}
\|q(t)\|_{H^{s}}\leq\|q_{0}\|_{H^{s}}+C_{s}^{2}\int_{0}^{t}\|q(t')\|_{X^{s}}^{3}dt'.\label{bound-1}\end{equation}

To estimate the $L^{2}$ norm of $p(t)$, we use the integral representation
\[
p(t)=P(t)-\int_{0}^{t}Le^{(t-t')L}f(q(t'))p(t')dt',\]
 where $P(t)=LQ(t)$ is defined by solution of the same linear problem
(\ref{Cauchy-linear}) with initial data $P(0)=p_{0}$. By the triangle
inequality and the norm-preservation property, we obtain \begin{eqnarray*}
\|p(t)\|_{L^{2}} & \leq & \|P(t)\|_{L^{2}}+\int_{0}^{t}\|Le^{(t-t')L}f(q(t'))p(t')\|_{L^{2}}dt',\\
 & \leq & \|p_{0}\|_{L^{2}}+\int_{0}^{t}\|Le^{(t-t')L}f(q(t'))p(t')\|_{L^{2}}dt'.\end{eqnarray*}
 The norm preservation (\ref{norm-preservation}) is not useful for
the second term since $Lf(q(t))p(t)$ may not be in $L^{2}$. On the
other hand, using formula (\ref{formula-P}), we write \begin{equation}
Le^{(t-t')L}f(q(t'))p(t')=-\int_{y}^{\infty}J_{t-t'}(y'-y)f(q(y',t'))p(y',t')dy'.\label{expression-f-p}\end{equation}
 We shall use the Hausdorf-Young's inequality \begin{equation}
\|f\star g\|_{L^{p}}\leq\|f\|_{L^{q}}\|g\|_{L^{r}},\label{Young}\end{equation}
 where $p,q,r$ are related by the constraint $q^{-1}+r^{-1}=1+p^{-1}$
and the star denotes convolution operator $f\star g=\int_{\mathbb{R}}f(y')g(y-y')dy'$.
Using inequality (\ref{Young}) and Lemma \ref{lemma-bounds}, we
obtain \[
\|Le^{(t-t')L}f(q(t'))p(t')\|_{L^{2}}\leq\|J_{t-t'}\|_{L^{\infty}}\|f(q(t'))p(t')\|_{L^{2/3}}\leq\|f(q(t'))p(t')\|_{L^{2/3}}.\]
 Using the Hï¿½lder inequality, we obtain \[
\|f(q(t))p(t)\|_{L^{1}}\leq\|f(q(t))\|_{L^{\rho}}\|p(t)\|_{L^{r}},\]
 with $\rho^{-1}+r^{-1}=1$, so that \[
\|Le^{(t-t')L}f(q(t'))p(t')\|_{L^{2}}\leq\|f(q(t'))\|_{L^{2\rho/3}}\|p(t')\|_{L^{2r/3}}.\]
 If we choose $r=3$, then we have $\rho=\frac{3}{2}$ and $\|f(q)\|_{L^{1}}\leq\|q\|_{L^{2}}^{2}$.
As a result, we conclude that \begin{equation}
\|p(t)\|_{L^{2}}\leq\|p_{0}\|_{L^{2}}+\int_{0}^{t}\|q(t')\|_{X^{s}}^{3}dt'.\label{bound-2}\end{equation}

Finally, we estimate the $L^{\infty}$ norm of $q(t)$ from the integral
equation (\ref{integral-form}). We obtain \begin{eqnarray*}
\|q(t)\|_{L^{\infty}} & \leq & \|Q(t)\|_{L^{\infty}}+\int_{0}^{t}\|e^{(t-t')L}f(q(t'))p(t')\|_{L^{\infty}}dt'.\end{eqnarray*}
 By Lemma \ref{lemma-bounds}, the convolution formula (\ref{solution-linear}),
and the Hausdorf-Young's inequality (\ref{Young}), the free term
is estimated by \begin{eqnarray*}
\|Q(t)\|_{L^{\infty}}\leq\|q_{0}\|_{L^{\infty}}+\|K_{t}\|_{L^{2}}\|q_{0}\|_{L^{2}}\leq\|q_{0}\|_{L^{\infty}}+C_{1}t^{1/2}\|q_{0}\|_{X^{s}},\end{eqnarray*}
 for some $C_{1}>0$. The nonlinear term is estimated by \begin{eqnarray*}
\|e^{(t-t')L}f(q(t'))p(t')\|_{L^{\infty}} & \leq & \|f(q(t'))p(t')\|_{L^{\infty}}+\|K_{t-t'}\|_{L^{\infty}}\|f(q(t'))p(t')\|_{L^{1}}\\
 & \leq & \|q(t')\|_{L^{\infty}}^{2}\|p(t')\|_{L^{\infty}}+\|K_{t-t'}\|_{L^{\infty}}\|q(t')\|_{L^{\infty}}\|q(t')\|_{L^{2}}\|p(t')\|_{L^{2}}\\
 & \leq & C_{2}(1+(t-t'))\|q(t')\|_{X^{s}}^{3},\end{eqnarray*}
 for some $C_{2}>0$. As a result, we conclude that \begin{equation}
\|q(t)\|_{L^{\infty}}\leq\|q_{0}\|_{L^{\infty}}+C_{1}t^{1/2}\|q_{0}\|_{X^{s}}+C_{2}\int_{0}^{t}(1+(t-t'))\|q(t')\|_{X^{s}}^{3}dt'.\label{bound-3}\end{equation}
 Using (\ref{bound-1}), (\ref{bound-2}), and (\ref{bound-3}), we
can see that there exists $T=T(s,q_{c},\delta,\alpha)>0$ such that
the vector field of the integral equation (\ref{integral-form}) is
a closed map of a finite-radius ball in $X_{c}^{s}$ to itself. The
value of $T>0$ satisfies the bounds \begin{eqnarray*}
\alpha+T(1+C_{s}^{2})\delta^{2} & \leq & 1,\\
\alpha q_{c}+C_{1}T^{1/2}\alpha\delta+\frac{1}{2}C_{2}T(T+2)\delta^{3} & \leq & q_{c}.\end{eqnarray*}
 Moreover, since $f(q)$ behaves like a quadratic function, a similar
analysis shows that the map is Lipschitz with respect to $q$ and
it is a contraction if $T>0$ is sufficiently small. Existence of
a unique fixed point of the integral equation (\ref{integral-form})
in a complete space $C([0,T],X_{c}^{s})$ follows by the contraction
mapping principle (see, e.g. \cite{Zeidel}). \end{proof}

\begin{corollary} Under the conditions of Theorem \ref{theorem-local-wellposedness},
we actually have $q(t)\in C([0,T],X_{c}^{s})\cap C^{1}([0,T],H^{s})$.
\label{corollary-C1} \end{corollary}

\begin{proof} The proof is based on the identity $q_{t}=\sqrt{1-q^{2}}p$
and the fact that $p(t)\in C([0,T],H^{s+1})$. \end{proof}

\begin{remark} Existence of a unique solution can be proved more
easily in a weaker space \[
\tilde{X}_{c}^{s}=\left\{ q\in H^{s},\;\; p\in L^{\infty}:\;\|q\|_{L^{\infty}}\leq q_{c}<1\right\} ,\quad s>\frac{1}{2},\]
 provided that $q_{0}\in X_{c}^{s}$. Since $p\in H^{1}$ if $q,p\in L^{2}$,
$X_{c}^{s}$ is continuously embedded to $\tilde{X}_{c}^{s}$. The
space $X_{c}^{s}$ turns out to be more suitable if we are to use
conserved quantities of the sine--Gordon equation. \end{remark}

\section{Correspondence between short-pulse and sine--Gordon equations}

We start by stating the local well-posedness theorem for the short-pulse
equation (\ref{short-pulse}) from \cite{SW04}.

\begin{theorem}{[}Schäfer \& Wayne, 2004] Assume that $u_{0}\in H^{2}$.
There exists a $T>0$ such that the short-pulse equation (\ref{short-pulse})
admits a unique solution \[
u(t)\in C([0,T],H^{2})\cap C^{1}([0,T],H^{1})\]
 satisfying $u(0)=u_{0}$. Furthermore, the solution $u(t)$ depends
continuously on $u_{0}$. \label{theorem-wayne} \end{theorem}

We can now compare the results following from Theorem \ref{theorem-local-wellposedness}
with Theorem \ref{theorem-wayne}. Using the transformation (\ref{transformation})
and setting $q=\sin(w)$ and $p_{y}=q$, we have \begin{equation}
\begin{array}{lcl}
u & = & w_{t}=\frac{q_{t}}{\sqrt{1-q^{2}}}=p,\\
u_{x} & = & \frac{w_{ty}}{\cos(w)}=\tan(w)=\frac{p_{y}}{\sqrt{1-q^{2}}},\\
u_{xx} & = & \frac{w_{y}}{\cos^{3}w}=\frac{p_{yy}}{\left(1-q^{2}\right)^{2}}.\end{array}\label{transformation-u-w}\end{equation}

If $q(t)\in X_{c}^{1}$ for all $t\in[0,T]$, there exists a uniform
bound $q_{c}\in(0,1)$ such that $\|q(t)\|_{L^{\infty}}\leq q_{c}$
for all $t\in[0,T]$. As a result, the $H^{2}$ norm on $u$ in $x$
is equivalent to the $H^{2}$ norm on $p$ in $y$ and the $H^{1}$
norm on $q$ in $y$, since $p_{y}=q$. The following lemma summarizes
on the correspondence.

\begin{lemma} \label{lemma-correspondence} Assume that $\|q\|_{L^{\infty}}\leq q_{c}<1$
and consider transformations (\ref{transformation}) and (\ref{transformation-u-w}).
There exist $C^{\pm}>0$ such that \[
\forall u,p\in H^{2}(\mathbb{R}):\quad C^{-}\|p(t)\|_{H^{2}}\leq\|u(t)\|_{H^{2}}\leq C^{+}\|p(t)\|_{H^{2}}.\]
 \end{lemma}

\begin{proof} The proof is given by direct computations, e.g. \begin{eqnarray*}
\sqrt{1-q_{c}^{2}}\|p\|_{L^{2}}^{2} & \leq\|u\|_{L^{2}}^{2}\leq & \|p\|_{L^{2}}^{2},\\
\|\partial_{y}p\|_{L^{2}}^{2} & \leq\|\partial_{x}u\|_{L^{2}}^{2}\leq & \frac{1}{\sqrt{1-q_{c}^{2}}}\|\partial_{y}p\|_{L^{2}}^{2},\\
\|\partial_{y}^{2}p\|_{L^{2}}^{2} & \leq\|\partial_{x}^{2}u\|_{L^{2}}^{2}\leq & \frac{1}{\sqrt{\left(1-q_{c}^{2}\right)^{7}}}\|\partial_{y}^{2}p\|_{H^{2}}^{2},\end{eqnarray*}
 where $q_{c}<1$ by the assumption of the lemma. \end{proof}

Combining Theorems \ref{theorem-local-wellposedness} and \ref{theorem-wayne},
we obtain a more precise result on local well-posedness of the short-pulse
and sine--Gordon equations.

\begin{theorem} Let $q(t)\in C([0,T_{1}],X_{c}^{1})\cap C^{1}([0,T],H^{1})$
be a solution of the sine--Gordon equation in Theorem \ref{theorem-local-wellposedness}
and Corollary \ref{corollary-C1} and $T_{1}>0$. Let $u(t)\in C([0,T_{2}],H^{2})\cap C^{1}([0,T_{2}],H^{1})$
be a solution of the short-pulse equation in Theorem \ref{theorem-wayne}
for some $T_{2}>0$. Let $q_{0}$, $p_{0}$ and $u_{0}$ be related
by the transformations (\ref{transformation}) and (\ref{transformation-u-w}).
Then, in fact, $p(t)\in C^{1}([0,T],H^{2})$ and $u(t)\in C^{1}([0,T],H^{2})$
for $T=\min(T_{1},T_{2})$, where $p_{y}=q$. \label{theorem-joint}
\end{theorem}

\begin{proof} If $q(t)\in X_{c}^{1}$ on $[0,T_{1}]$, then the bound
$\|q(t)\|_{L^{\infty}}\leq q_{c}$ holds on $[0,T_{1}]$ for some
$q_{c}\in(0,1)$. By Lemma \ref{lemma-correspondence} and Corollary
\ref{corollary-C1}, if $p(t)\in C([0,T_{1}],H^{2})$ then $u(t)\in C([0,T_{1}],H^{2})$
and if $q(t)\in C^{1}([0,T_{1}],H^{1})$ then $u_{x}\in C^{1}([0,T_{1}],H^{1})$.
The first assertion recovers the result of Theorem \ref{theorem-wayne},
while if $T=\min(T_{1},T_{2})$, the second assertion combing with
$u(t)\in C^{1}([0,T],H^{1})$ from Theorem \ref{theorem-wayne} implies
that $u(t)\in C^{1}([0,T],H^{2})$.

In the opposite direction, by Lemma \ref{lemma-correspondence}, if
$u(t)\in C([0,T_{2}],H^{2})\cap C^{1}([0,T_{2}],H^{1})$, then $p(t)\in C([0,T],H^{2})\cap C^{1}([0,T],H^{1})$
for $T=\min(T_{1},T_{2})$. Combining with $q(t)\in C^{1}([0,T],H^{1})$
from Corollary \ref{corollary-C1}, we obtain that $p(t)\in C^{1}([0,T],H^{2})$.
\end{proof}

\begin{remark} Theorem \ref{theorem-joint} shows that the results
on the sine--Gordon equation (\ref{sine-Gordon}) allow us to control
the $C^{1}$ property of $\|\partial_{x}^{2}u\|_{L^{2}}$ in the short-pulse
equation (\ref{short-pulse}), while the results on the short-pulse
equation (\ref{short-pulse}) allow us to control the $C^{1}$ property
of $\|p\|_{L^{2}}$ in the sine--Gordon equation (\ref{sine-Gordon}).
This duality turns out to be useful for rigorous treatment of the
conserved quantities for each of the two equations. \end{remark}

\begin{remark} If $u(t)$ is a solution of the short-pulse equation
in Theorem \ref{theorem-wayne}, then for all $t\in(0,T)$, we have
the zero-mass constraint \begin{eqnarray*}
\int_{\mathbb{R}}udx & = & \int_{\mathbb{R}}w_{t}\cos(w)dy=\frac{d}{dt}\int_{\mathbb{R}}\sin(w)dy=\frac{d}{dt}\int_{\mathbb{R}}qdy\\
 & = & \frac{d}{dt}\left[\lim_{y\to\infty}p(y,t)-\lim_{y\to-\infty}p(y,t)\right]=0,\end{eqnarray*}
 thanks to the fact that $p\in C^{1}([0,T],H^{2})$ from Theorem \ref{theorem-joint}.
We note that the initial data $u_{0}\in H^{2}$ does not have to satisfy
the zero-mass constraint $\int_{\mathbb{R}}u_{0}dx=0$, in which case
$\int_{\mathbb{R}}u(x,t)dx$ jumps from a nonzero value to zero instantaneously
for any $t>0$. See Ablowitz \& Villaroel \cite{AV} for analysis
of a similar problem in the context of the Kadomtsev--Petviashvili
equation. \end{remark}

\section{Global well-posedness of the short-pulse equation (\ref{short-pulse})}

It follows from the method of Picard iterations that the existence
time $T>0$ in Theorems \ref{theorem-wayne} and \ref{theorem-joint}
is inverse proportional to the norm $\|u_{0}\|_{H^{s}}$ of the initial
data. To prove Theorem \ref{theorem-wellposedness}, we need to control
the norm $\|u(T)\|_{H^{2}}$ by a $T$-independent constant. This
constant will be found from the values of conserved quantities of
the short-pulse equation. Formal computations of an infinite hierarchy
of conserved quantities were reported by Brunelli \cite{Br05}. Using
Theorem \ref{theorem-joint}, we shall make a rigorous use of the
conserved quantities.

\begin{lemma} Let $u(t)\in C^{1}([0,T],H^{2})$ be a solution of
the short-pulse equation (\ref{short-pulse}). The following integral
quantities are constant on $[0,T]$: \begin{eqnarray*}
H_{-1} & = & \int_{\mathbb{R}}u^{2}dx,\\
H_{0} & = & \int_{\mathbb{R}}\left(\sqrt{1+u_{x}^{2}}-1\right)dx=\int_{\mathbb{R}}\frac{u_{x}^{2}}{1+\sqrt{1+u_{x}^{2}}}dx,\\
H_{1} & = & \int_{\mathbb{R}}\sqrt{1+u_{x}^{2}}\left[\partial_{x}\left(\frac{u_{x}}{\sqrt{1+u_{x}^{2}}}\right)\right]^{2}dx=\int_{\mathbb{R}}\frac{u_{xx}^{2}}{(1+u_{x}^{2})^{5/2}}dx.\end{eqnarray*}
 \label{lemma-conserved-quantities-short-pulse} \end{lemma}

\begin{proof} We shall write the balance equations for the densities
of $H_{-1}$, $H_{0}$, and $H_{1}$: \begin{eqnarray*}
\partial_{t}\left(u^{2}\right) & = & \partial_{x}\left(v^{2}+\frac{1}{4}u^{4}\right),\\
\partial_{t}\left(\sqrt{1+u_{x}^{2}}-1\right) & = & \frac{1}{2}\partial_{x}\left(u^{2}\sqrt{1+u_{x}^{2}}\right),\\
\partial_{t}\left(\frac{u_{xx}^{2}}{\sqrt{(1+u_{x}^{2})^{5}}}\right) & = & \partial_{x}\left(\frac{2u_{x}^{2}}{\sqrt{1+u_{x}^{2}}}-\frac{u^{2}u_{xx}^{2}}{2\sqrt{(1+u_{x}^{2})^{5}}}\right),\end{eqnarray*}
 where $v=\partial_{x}^{-1}u=u_{t}-\frac{1}{2}u^{2}u_{x}$ thanks
to the short-pulse equation (\ref{short-pulse}). If $u(t)\in C^{1}([0,T],H^{2})$,
then $v(t)\in C([0,T],H^{1})$. By Sobolev's embedding, we have $v(t),u(t),u_{x}(t)\in L^{\infty}$
and $v(t),u(t),u_{x}(t)\to0$ as $|x|\to\infty$ for any $t\in[0,T]$.
Integrating the first two balance equations on $\mathbb{R}$, we confirm
conservation of $H_{-1}$ and $H_{0}$. To prove conservation of $H_{1}$,
we need to show that $uu_{xx}\to0$ as $|x|\to\infty$ for any $t\in[0,T]$.
Using (\ref{short-pulse}) and (\ref{transformation}), we obtain
\[
\frac{1}{2}uu_{xx}-u_{x}^{2}=\frac{u_{xt}}{u}-1=\tan^{2}(w)=\frac{q^{2}}{1-q^{2}},\]
 where $u_{x}\to0$ as $|x|\to\infty$ and $q=q(y,t)$. By Theorem
\ref{theorem-joint}, $q(t)\in C^{1}([0,T],H^{1})$ and $\|q(t)\|_{L^{\infty}}\leq q_{c}<1$
for any $t\in[0,T]$. Therefore, \[
\frac{\partial x}{\partial y}=\cos(w)=\sqrt{1-q^{2}}\geq\sqrt{1-q_{c}^{2}}>0,\]
 for any $t\in[0,T]$, so that the limits $y\to\pm\infty$ correspond
to the limits $x\to\pm\infty$. Furthermore, since $q\to0$ as $|y|\to\infty$,
we have $uu_{xx}\to0$ as $|x|\to\infty$. \end{proof}

We can now prove Theorem \ref{theorem-wellposedness}.

\begin{proof1}{\em of Theorem \ref{theorem-wellposedness}.} The
values of $H_{-1}$, $H_{0}$ and $H_{1}$ computed at initial data
$u_{0}\in H^{2}$ are bounded by \begin{eqnarray*}
H_{-1} & = & \int_{\mathbb{R}}u^{2}dx\leq\|u_{0}\|_{H^{2}}^{2},\\
H_{0} & = & \int_{\mathbb{R}}\frac{u_{x}^{2}}{1+\sqrt{1+u_{x}^{2}}}dx\leq\frac{1}{2}\|u_{0}\|_{H^{2}}^{2},\\
H_{1} & = & \int_{\mathbb{R}}\frac{u_{xx}^{2}}{(1+u_{x}^{2})^{5/2}}dx\leq\|u_{0}\|_{H^{2}}^{2}.\end{eqnarray*}
 Note that if $\|u_{0}'\|_{L^{2}}^{2}+\|u_{0}''\|_{L^{2}}^{2}<1$,
then $2H_{0}+H_{1}<1$. By Lemma \ref{lemma-conserved-quantities-short-pulse},
these quantities remain constant on $[0,T]$. We will show that the
quantities $H_{0}$ and $H_{1}$ give an upper bound for $H^{1}$
norm of the variable \begin{equation}
\tilde{q}=\frac{u_{x}}{\sqrt{1+u_{x}^{2}}}.\label{map-q-ux}\end{equation}
 Note that $\tilde{q}(x,t)=q(y,t)=\sin(w(y,t))$, where $x=x(y,t)$
is defined by the transformation (\ref{transformation}). To control
$\|\tilde{q}\|_{H^{1}}$, we obtain \begin{eqnarray*}
\int_{\mathbb{R}}\tilde{q}^{2}dx & = & \int_{\mathbb{R}}\frac{u_{x}^{2}}{1+u_{x}^{2}}dx=\int_{\mathbb{R}}\frac{u_{x}^{2}}{1+\sqrt{1+u_{x}^{2}}}\frac{1+\sqrt{1+u_{x}^{2}}}{1+u_{x}^{2}}dx\\
 & \leq & 2\int_{\mathbb{R}}\frac{u_{x}^{2}}{1+\sqrt{1+u_{x}^{2}}}dx=2H_{0}\end{eqnarray*}
 and \begin{eqnarray*}
\int_{\mathbb{R}}\tilde{q}_{x}^{2}dx & = & \int_{\mathbb{R}}\left[\partial_{x}\left(\frac{u_{x}}{\sqrt{1+u_{x}^{2}}}\right)\right]^{2}dx\\
 & \leq & \int_{\mathbb{R}}\sqrt{1+u_{x}^{2}}\left[\partial_{x}\left(\frac{u_{x}}{\sqrt{1+u_{x}^{2}}}\right)\right]^{2}dx=H_{1}.\end{eqnarray*}
 If $u(t)\in C([0,T],H^{2})$, then $q(t)\in C([0,T],H^{1})$ and
$\tilde{q}(t)$ satisfies the $T$-independent bound \[
\|\tilde{q}(t)\|_{H^{1}}\leq\sqrt{H_{1}+2H_{0}}<1,\quad\forall t\in[0,T].\]
Thanks to Sobolev's embedding $\|\tilde{q}\|_{L^{\infty}}\leq\frac{1}{\sqrt{2}}\|\tilde{q}\|_{H^{1}}$,
we have $\|\tilde{q}(t)\|_{L^{\infty}}\leq q_{c}:=\frac{1}{\sqrt{2}}\sqrt{H_{1}+2H_{0}}<1$,
$\forall t\in[0,T]$. Inverting the map (\ref{map-q-ux}), we obtain
\[
u_{x}=\frac{\tilde{q}}{\sqrt{1-\tilde{q}^{2}}}.\]
 Since $H^{1}$ is a Banach algebra with $C_{1}=1$ in the bound (\ref{banach-algebra})
(see, e.g., \cite{Morosi}), we expand the map into Taylor series
with positive coefficients to obtain \[
\|u_{x}\|_{H^{1}}\leq\frac{\|\tilde{q}\|_{H^{1}}}{\sqrt{1-\|\tilde{q}\|_{H^{1}}^{2}}},\]
 which results in the $T$-independent bound \[
\|u(T)\|_{H^{2}}\leq\left(H_{-1}+\frac{H_{1}+2H_{0}}{1-(H_{1}+2H_{0})}\right)^{1/2}.\]
 This bound allows us to choose a constant time step $T_{0}$ such
that the solution $u(T_{0})$ can be continued on the interval $[T_{0},2T_{0}]$
in space $C([T_{0},2T_{0}],H^{2})$ using the same Theorems \ref{theorem-wayne}
and \ref{theorem-joint}. Continuing the solution with a uniform time
step $T_{0}>0$, we obtain global existence of solutions in space
$u(t)\in C(\mathbb{R}_{+},H^{2})$, which completes the proof of Theorem
\ref{theorem-wellposedness}. \end{proof1}

The sufficient condition for global well-posedness of Theorem \ref{theorem-wellposedness}
can be sharpen thanks to the scaling invariance of the short-pulse
equation (\ref{short-pulse}). Let $\alpha\in\mathbb{R}_{+}$ be an
arbitrary parameter. If $u(x,t)$ is a solution of (\ref{short-pulse}),
then $U(X,T)$ is also a solution of (\ref{short-pulse}) with \begin{equation}
X=\alpha x,\quad T=\alpha^{-1}t,\quad U(X,T)=\alpha u(x,t).\label{scaling-invariance}\end{equation}
 The scaling invariance yields the following transformation for conserved
quantities \begin{eqnarray*}
\tilde{H}_{0} & = & \int_{\mathbb{R}}\left(\sqrt{1+U_{X}^{2}}-1\right)dX=\alpha\int_{\mathbb{R}}\left(\sqrt{1+u_{x}^{2}}-1\right)dx=\alpha H_{0},\\
\tilde{H}_{1} & = & \int_{\mathbb{R}}\frac{U_{XX}^{2}}{(1+U_{X}^{2})^{5/2}}dX=\alpha^{-1}\int_{\mathbb{R}}\frac{u_{xx}^{2}}{(1+u_{x}^{2})^{5/2}}dx=\alpha^{-1}H_{1}.\end{eqnarray*}
 Therefore, for a given $u_{0}\in H^{2}$, we obtain a family of initial
data $U_{0}\in H^{2}$ satisfying \[
\phi(\alpha)=2\tilde{H}_{0}+\tilde{H}_{1}=2\alpha H_{0}+\alpha^{-1}H_{1}.\]
 Function $\phi(\alpha)$ achieves its minimum of $2\sqrt{2H_{0}H_{1}}$
at $\alpha=\sqrt{\frac{H_{2}}{2H_{1}}}$, so that \[
\forall\alpha\in\mathbb{R}_{+}:\quad\phi(\alpha)\geq2\sqrt{2H_{0}H_{1}}.\]
 Using the scaling invariance property, we obtain the following corollary
to Theorem \ref{theorem-wellposedness}.

\begin{corollary} Assume that $u_{0}\in H^{2}$ and $2\sqrt{2H_{0}H_{1}}<1$.
Then the short-pulse equation (\ref{short-pulse}) admits a unique
solution $u(t)\in C(\mathbb{R}_{+},H^{2})$ satisfying $u(0)=u_{0}$.
\end{corollary}

\begin{proof} If $u_{0}\in H^{2}$ satisfies $2\sqrt{2H_{0}H_{1}}<1$,
there exists $\alpha\in\mathbb{R}_{+}$ such that the corresponding
$U_{0}\in H^{2}$ satisfies $2\tilde{H}_{0}+\tilde{H}_{1}<1$. By
Theorem \ref{theorem-wellposedness}, the corresponding solution $U(T)\in C(\mathbb{R}_{+},H^{2})$,
so that $u(t)\in C(\mathbb{R}_{+},H^{2})$ by the scaling transformation
(\ref{scaling-invariance}). \end{proof}

\section{Global well-posedness of the sine--Gordon equation (\ref{sine-Gordon})}

The sine--Gordon equation (\ref{sine-Gordon}) has an infinite set
of conserved quantities similarly to the short-pulse equation (\ref{short-pulse}).
These conserved quantities can be enumerated by the order $j\geq0$
in the term $(\partial_{y}^{j}w)^{2}$ involving the highest spatial
derivative. We will use only the first two conserved quantities, \[
E_{0}=\int_{\mathbb{R}}(1-\cos(w))dy,\quad E_{1}=\int_{\mathbb{R}}w_{y}^{2}dy,\]
 existence of which follow formally from the balance equations \begin{eqnarray*}
\partial_{t}\left(1-\cos(w)\right)=\partial_{y}\left(\frac{1}{2}w_{t}^{2}\right),\quad\partial_{t}\left(\frac{1}{2}w_{y}^{2}\right)=\partial_{y}\left(1-\cos(w)\right).\end{eqnarray*}
 Additionally, the sine--Gordon equation (\ref{sine-Gordon}) has
another infinite set of conserved quantities involving trigonometric
functions of $w$ and their integrals enumerated by $j\leq0$ in the
term $(\partial_{y}^{j}w)^{2}$. Besides $E_{0}$, we need only one
conserved quantity of this set, \[
E_{-1}=\int_{\mathbb{R}}\cos(w)w_{t}^{2}dy,\]
 existence of which follows formally from the balance equation \[
\partial_{t}\left(\cos(w)w_{t}^{2}\right)=\partial_{y}\left(w_{tt}^{2}-\frac{1}{4}w_{t}^{4}\right).\]
 Using the transformation $q=\sin(w)$, we rewrite the conserved quantities
in the equivalent form \begin{equation}
E_{-1}=\int_{\mathbb{R}}\sqrt{1-q^{2}}p^{2}dy,\quad E_{0}=\int_{\mathbb{R}}f(q)dy,\quad E_{1}=\int_{\mathbb{R}}\frac{q_{y}^{2}}{1-q^{2}}dy,\label{conserved-quantities}\end{equation}
 where $p=\partial_{y}^{-1}q$ and $f(q)$ is defined by (\ref{nonlinear-function}).
The balance equations are rewritten in the corresponding forms \[
\partial_{t}f(q)=\partial_{y}\left(\frac{1}{2}p^{2}\right),\quad\partial_{t}\left(\frac{q_{y}^{2}}{1-q^{2}}\right)=\partial_{y}f(q)\]
 and \[
\partial_{t}\left(\sqrt{1-q^{2}}p^{2}\right)=\partial_{y}\left(p_{t}^{2}-\frac{1}{4}p^{4}\right).\]
 We shall check if $E_{1}$, $E_{0}$, and $E_{-1}$ are time conserved
quantities for the Cauchy problem (\ref{Cauchy}). Global well-posedness
in $H^{2}$ follows from analysis of the three conserved quantities.

\begin{lemma} \label{lemma-conserved-quantities} Let $p(t)\in C^{1}([0,T],H^{2})$
be the solution of the Cauchy problem (\ref{Cauchy}) and $q(t)=\partial_{y}p(t)$.
Then, $E_{1}$, $E_{0}$, and $E_{-1}$ are constant on $[0,T]$.
\end{lemma}

\begin{proof} By Sobolev's embedding for $p(t)\in C^{1}([0,T],H^{2})$,
we have \textbf{$q(t)$, $p(t)$,} $p_{t}(t)\to0$ as $|y|\to\infty$.
Therefore, conservation of $E_{1}$, $E_{0}$, and $E_{-1}$ follows
by integrating the balance equations on $y\in\mathbb{R}$ for the
local solutions. \end{proof}

\begin{theorem} \label{theorem-global-wellposedness} Assume that
$q_{0}\in X_{c}^{1}$ and $2E_{0}+E_{1}<1$. Then there exist a unique
global solution $q(t)\in C(\mathbb{R}_{+},X_{c}^{1})$ of the Cauchy
problem (\ref{Cauchy}) satisfying $q(0)=q_{0}$. \end{theorem}

\begin{proof} The values of $E_{-1},E_{0},E_{1}$ computed at initial
data $q_{0}\in X_{c}^{1}$ are bounded by \[
E_{-1}\leq\|p_{0}\|_{L^{2}}^{2},\quad|E_{0}|\leq\|q_{0}\|_{L^{2}}^{2},\quad E_{1}\leq\frac{1}{1-\|q_{0}\|_{L^{\infty}}^{2}}\|q_{0}'\|_{L^{2}}^{2},\]
 where the constraint $\|q_{0}\|_{L^{\infty}}\leq q_{c}<1$ is used.
By Lemma \ref{lemma-conserved-quantities}, if $q(t)\in C^{1}([0,T],X_{c}^{1})$
is a solution constructed in Theorems \ref{theorem-local-wellposedness}
and \ref{theorem-joint} for a fixed $T>0$, the values of quantities
$E_{-1},E_{0},E_{1}$ are constant on $[0,T]$. Therefore, we only
need to bound the norm $\|q(t)\|_{X^{1}}=\|q(t)\|_{H^{1}}+\|p(t)\|_{L^{2}}$
by a combination of $E_{-1},E_{0},E_{1}$. This bound is obtained
from the following estimates \[
E_{-1}\geq\|p(t)\|_{L^{2}}^{2}\sqrt{1-\|q(t)\|_{L^{\infty}}^{2}},\quad E_{0}\geq\frac{1}{2}\|q(t)\|_{L^{2}}^{2},\quad E_{1}\geq\|\partial_{y}q(t)\|_{L^{2}}^{2},\quad\forall t\in[0,T].\]
 By Sobolev's embedding and the bounds above, we have \[
\|q(t)\|_{L^{\infty}}\leq\frac{1}{\sqrt{2}}\|q(t)\|_{H^{1}}\leq q_{c}:=\frac{1}{\sqrt{2}}\sqrt{E_{1}+2E_{0}}<1,\]
 since $E_{1}+2E_{0}<1$ thanks to the assumption of the theorem.
As a result, we obtain the bound \[
\|q(t)\|_{X^{1}}\leq\sqrt{E_{1}+2E_{0}}+\sqrt{\frac{E_{-1}}{\sqrt{1-q_{c}^{2}}}},\quad\forall t\in[0,T].\]
 The time step $T>0$ depends on $\|q_{0}\|_{X^{1}}$. Since the above
norm is bounded by the $T$-independent constant on $[0,T]$, one
can choose a non-zero time step $T_{0}>0$ such that the solution
$q(t)$ can be continued on the interval $[T_{0},2T_{0}]$ using the
same Theorems \ref{theorem-local-wellposedness} and \ref{theorem-joint}.
Continuing the solution with a uniform time step $T_{0}$, we obtain
global existence of solutions $q(t)\in C(\mathbb{R}_{+},X_{c}^{1})$.
\end{proof}

\begin{remark} Theorem \ref{theorem-global-wellposedness} is very
similar to Theorem \ref{theorem-wellposedness} thanks to correspondence
between the two equations in Lemma \ref{lemma-correspondence}. In
particular, it follows directly that $H_{1}=E_{1}$, $H_{0}=E_{0}$
and $H_{-1}=E_{-1}$. \end{remark}

\end{document}